\documentclass[11pt]{amsart}
\usepackage{mathrsfs}
\usepackage{amssymb}

\usepackage{amscd}
\usepackage{amsmath, amssymb}
\usepackage{amsfonts}
\usepackage[colorlinks,linkcolor=blue,citecolor=blue, pdfstartview=FitH]{hyperref}

\numberwithin{equation}{section}

\theoremstyle{plain}
\newtheorem{thm}{Theorem}[section]
\newtheorem{theorem}[thm]{Theorem}
\newtheorem{lemma}[thm]{Lemma}
\newtheorem{corollary}[thm]{Corollary}
\newtheorem{proposition}[thm]{Proposition}

\theoremstyle{definition}

\newtheorem{remark}[thm]{Remark}

\newtheorem{definition}[thm]{Definition}

\newtheorem{example}[thm]{Example}

\newtheorem{defn-thm}[thm]{Definition-Theorem}

\newcommand{\sO}{{\mathcal O}}

\def\b{\bar}
\def\mb{\mathbb}
\def\mc{\mathcal}

\def\t{\Delta}

\renewcommand{\P}{{\mathbb P}}
\newcommand{\Q}{{\mathbb Q}}
\newcommand{\R}{{\mathbb R}}

\newcommand{\qtq}[1]{\quad\mbox{#1}\quad}
\newcommand{\bp}{\bar{\partial}}
\newcommand{\Om}{\Omega}

\newcommand{\ds}{\oplus}

\newcommand{\ts}{\otimes}

\newcommand{\btheorem}{\begin{theorem}}
\newcommand{\etheorem}{\end{theorem}}
\newcommand{\bproposition}{\begin{proposition}}
\newcommand{\eproposition}{\end{proposition}}
\newcommand{\bdefinition}{\begin{definition}}
\newcommand{\edefinition}{\end{definition}}
\newcommand{\bcorollary}{\begin{corollary}}
\newcommand{\ecorollary}{\end{corollary}}
\newcommand{\bproof}{\begin{proof}}
\newcommand{\eproof}{\end{proof}}
\newcommand{\bremark}{\begin{remark}}
\newcommand{\eremark}{\end{remark}}
\newcommand{\eexample}{\end{example}}
\newcommand{\bexample}{\begin{example}}

\newcommand{\elemma}{\end{lemma}}
\newcommand{\blemma}{\begin{lemma}}

\newcommand{\sq}{\sqrt{-1}}

\newcommand{\p}{\partial}

\renewcommand{\bar}{\overline}

\renewcommand{\phi}{\varphi}

\newcommand{\ka}{K\"ahler }

\newcommand{\ee}{\end{eqnarray*}}
\newcommand{\be}{\begin{eqnarray*}}

\newcommand{\beq}{\begin{equation}}
\newcommand{\eeq}{\end{equation}}

\newcommand{\bd}{\begin{enumerate}}
\newcommand{\ed}{\end{enumerate}}

\renewcommand{\tilde}{\widetilde}

\renewcommand{\>}{\rightarrow}

\usepackage{fancyhdr}
\begin{document}
\title{Logarithmic vanishing theorems on compact K\"{a}hler manifolds I}
\makeatletter
\let\uppercasenonmath\@gobble
\let\MakeUppercase\relax
\let\scshape\relax
\makeatother

\author{Chunle Huang}

\address{Chunle Huang, Center of Mathematical Science, Zhejiang
University, Hangzhou, 310027, CHINA. } \email{chunle@zju.edu.cn}

\author{Kefeng Liu}
\address{Kefeng Liu, Department of Mathematics,
Capital Normal University, Beijing, 100048, China}
\address{Department of Mathematics, University of California at Los Angeles, California 90095}

\email{liu@math.ucla.edu}
\thanks{1. K. Liu is supported in part by NSF Grant.}
\author{Xueyuan Wan}
\address{Xueyuan Wan, Chern Institute of Mathematics \& LPMC, Nankai University, 300071}
\email{xywan@mail.nankai.edu.cn}
\thanks{2. X. Wan is partially supported by China Scholarship Council / University of California, Los Angeles Joint PhD. Student. }

\author{Xiaokui Yang}\thanks{3. X.Yang is partially supported by China's
Recruitment
 Program of Global Experts and National Center for Mathematics and Interdisciplinary Sciences,
 Chinese Academy of Sciences.}
\address{Xiaokui Yang, Morningside Center of Mathematics, Academy of Mathematics and Systems Science\\ Chinese Academy of Sciences, Beijing, 100190, China}
\address{Hua Loo-Keng Key Laboratory of Mathematics, Academy of Mathematics and Systems Science\\ Chinese Academy of Sciences, Beijing,100190, China}

 \email{xkyang@amss.ac.cn}

\maketitle
\begin{abstract}
In this paper, we first establish an  $L^2$-type Dolbeault
isomorphism for logarithmic differential forms by H\"{o}rmander's
$L^2$-estimates. By using this isomorphism and the construction of
smooth Hermitian metrics, we obtain a number of new
vanishing theorems for sheaves of logarithmic differential forms on
compact K\"ahler manifolds with simple normal crossing divisors,
which generalize several classical vanishing theorems, including
Norimatsu's vanishing theorem, Gibrau's vanishing theorem, Le
Potier's vanishing theorem and a version of the Kawamata-Viehweg
vanishing theorem.
\end{abstract}

\maketitle 
%

\section{Introduction}
The basic properties of the sheaf of  logarithmic differential forms
and of the sheaves with logarithmic integrable connections on smooth
projective manifolds were developed  by Deligne in \cite{Del70}.
Esnault and Viehweg investigated in \cite{EV06} the relations
between logarithmic de Rham complexes and vanishing theorems on
complex algebraic manifolds, and showed that many vanishing theorems
follow from the degeneration of certain Hodge to de Rham type
spectral sequences. For a comprehensive description of the topic,
we refer the reader to Esnault and Viehweg's work \cite{EV92} and
also the references therein.

In this paper, we  develop an effective analytic method to prove
vanishing theorems for the sheaves of logarithmic differential forms
on \emph{compact K\"ahler manifolds}. One of our motivations to
develop this method is to prove Fujino's conjecture (e.g.
\cite[Conjecture~2.4]{Fujino2},
\cite[Problem~1.8]{Fujino3},\cite[Section~3]{Fujino4} and
\cite[Conjecture~1.2]{Mat16}) on a logarithmic version of the Kollar
injectivity theorem. Let $X$ be a compact K\"ahler manifold of
dimension $n$ and $Y=X-D$ where $D=\sum_{i=1}^{s}D_i$ is a simple
normal crossing divisor in $X$. Suppose that $E$ is an Hermitian
vector bundle over $X$. We first describe the key steps and main
difficulties in our analytic approach. Let $h_Y^E$ and $\omega_Y$ be
two smooth metrics on $E|_Y$ and $Y$ respectively, then we need to
show

\bd \item[(a)] there is an $L^2$ fine resolution
$(\Omega_{(2)}^{p,\bullet}(X,E,\omega_Y,{h}_Y^E), \bp)$
 of the sheaf of  logarithmic holomorphic differential forms $\Omega^{p}(\log
 D)\otimes\mathcal{O}(E)$ whenever the metrics $h_Y^E$ and $\omega_Y$ are chosen to be
 suitable;

\item[(b)] the desired curvature conditions for $h_Y^E$ and $\omega_Y$ can imply
vanishing theorems for
$(\Omega_{(2)}^{p,\bullet}(X,E,\omega_Y,{h}_Y^E), \bp)$ by using
$L^2$-estimate. \ed

\noindent The main difficulties arise from the construction of the
Hermitian metric $h^E_Y$ and the Poincar\'e type metric $\omega_Y$
which are suitable for both $(a)$ and $(b)$.

It is well-known that
various vanishing theorems are very important in complex
analytic geometry and algebraic geometry.  For instance,  the
Akizuki-Kodaira-Nakano vanishing theorem asserts that if $L$ is a
positive line bundle over a compact K\"ahler manifold $X$, then
\begin{equation}
H^{q}(X, \Om_X^p\ts L)=0 \qtq{for any $p+q\geq \dim X+1$.}\nonumber
\end{equation}
The main purpose of this paper is  to investigate logarithmic type
Akizuki-Kodaira-Nakano vanishing theorems for the pair ($X,D$). The
first main result of our paper is

\begin{thm} \label{main} Let $X$ be a compact \ka manifold of dimension $n$ and $D=\sum_{i=1}^s D_i$ be a
simple normal crossing divisor in $X$.  Let $N$ be a line bundle and
$\Delta=\sum_{i=1}^s a_iD_i$ be an $\R$-divisor  with $a_i\in[0,1]$
such that $N\otimes \mc{O}_X([\t])$ is a $k$-positive $\R$-line
bundle. Then for any nef line bundle $L$, we have
  \begin{align*}
    H^q(X,\Omega^p_X(\log D)\otimes L\otimes N)=0\quad \text{for any}\,\,\, p+q\geq n+k+1.
  \end{align*}
\end{thm}

\noindent As we pointed out before, the key ingredient  is the
construction of suitable Hermitian metrics.
In the analytical setting, the positivity of  $\R$-line bundle, which will be defined in Section 2.4, is defined
by using positivity of curvature which is very flexible to use, since we can multiply arbitrary real coefficients on the curvature
of a line bundle to obtain certain desired curvature property. In particular,
the theory of  $\R$-divisors (or
$\R$-line bundles) in algebraic geometry is not used in this paper, which is also a notable
advantage in our analytic approach. On the other hand, the setting
in Theorem \ref{main} is quite general and it has many
straightforward applications in complex analytic geometry and
complex algebraic geometry. The first application is the following
log type Gibrau's vanishing theorem.
\begin{corollary}
Let $X$ be a compact \ka manifold of dimension $n$ and $D$ be a
simple normal crossing divisor in $X$. If $L$ is a nef line bundle
and $N$ is a $k$-positive line bundle over $X$, then  $$  H^{q}(X,
\Om_X^p(\log D)\ts L\ts N)=0 \qtq{for any $p+q\geq n+k+1$.} $$
\end{corollary}

\noindent In particular, we have

\bcorollary\label{0c22} Let $X$ be a compact \ka manifold of
dimension $n$ and $D$ be a simple normal crossing divisor. Suppose
that $L\rightarrow X$ is an ample line bundle, then
\begin{equation}
H^{q}(X, \Om_X^p(\log D)\ts L)=0\qtq{for any $p+q\geq
n+1$.}\nonumber
\end{equation}
\ecorollary

\noindent This well-known result is proved  by Norimatsu
(\cite{Norimatsu19}) using analytic methods (see also
Deligne-Illusie's proof in \cite{DI87} by the
characteristic $p$ methods). As an analogue to Corollary \ref{0c22},
we  obtain  the following log  type Le Potier vanishing theorem for
ample vector bundles.

\bcorollary  Let $X$ be a compact \ka manifold of dimension $n$ and
$D$ be a simple normal crossing divisor. Suppose that $E\rightarrow
X$ is an ample vector  bundle of rank $r$. Then
\begin{equation}
H^{q}(X, \Om_X^p(\log D)\ts E)=0\qtq{for any $p+q\geq
n+r$.}\nonumber
\end{equation}
\ecorollary

\vskip 0.3\baselineskip

As we know that the Kawamata-Viehweg type vanishing theorems have
played fundamental roles in algebraic geometry and complex analytic
geometry (e.g. \cite{EV06, Dem, Cao14, GZ15, FM}). As another
application of Theorem \ref{main}, we get a log type vanishing
theorem for $k$-positive line bundles over compact \emph{K\"ahler
manifolds}, which generalizes a version of the Kawamata-Viehweg
vanishing theorem over projective manifolds.

\begin{theorem}\label{main2}
 Let $X$ be a compact \ka manifold of
dimension $n$ and $D=\sum_{i=1}^s D_i$ be a simple normal crossing
divisor. Suppose $F$ is a line bundle over $X$ and $m$ is a positive
real number such that $mF=L+D'$, where $D'=\sum_{i=1}^s \nu_i D_i$
is an effective normal crossing $\R$-divisor and  $L$ is a
$k$-positive $\R$-line bundle. Then
  \begin{align}
    H^q\left(X,\Omega^p(\log D)\otimes F\otimes
    \mc{O}\left(-\sum_{i=1}^s\left(1+\left[\frac{\nu_i}{m}\right]\right)D_i\right)\right)=0\label{KVk}
  \end{align}
  for $p+q\geq n+k+1$.
\end{theorem}

\noindent \textbf{Remark}. In particular, if $mF=L+D'$ where $L$ is
an ample line bundle and $D'$ is an effective divisor, bypassing
Hironaka's desingularization procedure, one  obtains the classical
Kawamata-Viehweg vanishing from \ref{KVk} by taking $p=n$ and $k=0$.
It is also worth mentioning that, in \cite[Theorem 6.1]{Huazhang16},
Luo obtained a version of logarithmic vanishing theorem under  the
$k$-ample condition over a smooth projective variety and his proof
relies on the hyperplane induction methods on projective manifold
(e.g. the existence of very ample divisors). It is apparently
different from our unified analytic approaches over \emph{K\"ahler
manifolds}. On the other hand, it is also pointed out in
\cite[p.~127]{SS} that, the $k$-ampleness is irrelevant to the
$k$-positivity when $1\leq k\leq \dim X$.

\vskip 0.3\baselineskip

  Theorem \ref{main2} has
several variants and the first one is

\begin{corollary}Let $X$ be a compact K\"{a}hler manifold
$D=\sum_{j=1}^sD_j$ be a simple normal crossing divisor of $X$. Let
$[D']$ be a $k$-positive $\mb{R}$-line bundle over $X$, where
$D'=\sum_{i=1}^s c_i D_i$ with $c_i>0$ and $c_i\in\mb{R}$. Then
$$H^q(X,\Omega^p(\log D)\otimes \mc{O}_X(-\lceil D'\rceil))=0 ~~\text{for any}~~ p+q<n-k.$$In particular, when $[D']$ is  ample, \beq H^q(X,\Omega^p(\log
D)\otimes \mc{O}_X(-\lceil D'\rceil))=0,\ ~~\text{for}~~
p+q<n.\label{KV} \eeq
\end{corollary}

\noindent \textbf{Remark.}
  By using Serre duality, one obtains a special case of (\ref{KV})
 $$H^q(X,K_X\ts \mc{O}_X(\lceil D'\rceil))=0, \  ~~\text{for}~~ q>0.$$
This is  proved by Kawamata in \cite[Theorem~1]{Ka2} (see also
\cite[Corollary~1-2-2]{Kawa1}, \cite[Theorem
3.1.7]{Fujino} and \cite[Theorem~5.1]{MO05}).\\

\noindent The second variant is

\bcorollary Let $X$ be a compact K\"{a}hler manifold and
$D=\sum_{j=1}^sD_j$ be a simple normal crossing divisor of $X$. Let
$[D']$ be a $k$-positive $\mb{R}$-line bundle over $X$, where
$D'=\sum_{i=1}^s a_i D_i$ with $a_i>0$ and $a_i\in\mb{R}$.
 If
there exists a line bundle $L$ over $X$ and a real number $b$ with
$0< a_j<b$ for all $j$
 and $bL=[D']$ as $\R$-line bundles. Then
  $$H^q(X,\Omega^p(\log D)\otimes L^{-1})=0$$
   for $p+q>n+k$ and $p+q<n-k$.

\ecorollary \noindent Note that, Esnault and Viehweg obtained  a
similar result in \cite[Theorem~6.2(a)]{EV06} for $\Q$-divisors by
using the the degeneration of the logarithmic Hodge to de Rham
spectral sequence together with the cyclic covering trick over projective
manifolds. The third variant is \bcorollary\label{corollary1.8} Let
$X$ be a compact \ka manifold of dimension $n$ and $D=\sum_{i=1}^s
D_i$ be a simple normal crossing divisor in $X$. Suppose there exist
some real constants $a_i\geq 0$ such that $\sum_{i=1}^s a_iD_i$
is a $k$-positive $\R$-divisor, then for any nef line bundle $L$, we
have
  \begin{align*}
    H^q(X,\Omega^p_X(\log D)\otimes L)=0\quad \text{for any}\,\,\, p+q\geq n+k+1.
  \end{align*}
\ecorollary

\noindent Note that Corollary \ref{corollary1.8} generalizes
\cite[Corollary~3.1.2]{Fujino}.

\noindent \textbf{Remark}. In a sequel to this paper, we will
systematically  investigate a number of vanishing theorems in
algebraic geometry by using analytic methods introduced in this
paper. For instance, we have obtained a version of Theorem
\ref{main} for $k$-ample line bundles on algebraic manifolds.

\vskip 1\baselineskip \textbf{Acknowledgements.}  The authors would
also like to thank Junyan Cao, Yifei Chen, Jean-Pierre Demailly,
Osamu Fujino, Shin-ichi Matsumura, Xiaotao Sun, Valentino Tosatti,
Jian Xiao and Xiangyu Zhou for their interests and/or discussions.

\vskip 1\baselineskip

\section{Preliminaries}
\subsection{Positivity of vector bundles}
Let $E$ be a holomorphic vector bundle of rank $r$ over a complex
manifold $M$ and $h$ be a smooth Hermitian metric on $E$. There
exists a unique connection $\nabla$, called the Chern connection of
$(E,h)$, which is compatible with the metric $h$ and complex
structure on $E$. Let $\{z^i\}_{i=1}^n$ be the local holomorphic
coordinates on $M$ and $\{e_{\alpha}\}_{\alpha=1}^r$ be the local
holomorphic frames of $E$. Locally, the curvature tensor of $(E,h)$
takes the form
\begin{equation}
  \sq \Theta(E, h) =\sq R_{i\overline{j}\alpha}^\gamma dz^i\wedge d\overline{z}^j \otimes e^{\alpha}\otimes e_{\gamma}\nonumber
\end{equation}
where $R_{i\overline{j}\alpha}^\gamma=h^{\gamma\overline{\beta}}R_{i\overline{j}\alpha\overline{\beta}}$ and\\
\beq
  R_{i\overline{j}\alpha\overline{\beta}}=-\frac{\partial^2 h_{\alpha\overline{\beta}}}{\partial z^i\partial \overline{z}^j}
  +h^{\gamma\overline{\delta}}\frac{\partial h_{\alpha\overline{\delta}}}{\partial z^i}\frac{\partial h_{\gamma\overline{\beta}}}{\partial \overline{z}^j}.
\label{cur}\eeq Here and henceforth we adopt the Einstein convention
for summation.

 \bdefinition

 An Hermitian vector bundle $(E,h)$ is said to be Griffiths-positive, if for any nonzero vectors $u=u^i \frac{\partial}{\partial z^i}$ and $v=v^\alpha e_\alpha,$
\begin{equation}
  \sum_{i,j,\alpha,\beta}R_{i\overline{j}\alpha\overline{\beta}}u^i\overline{u}^j v^\alpha \overline{v}^\beta>0.\nonumber
\end{equation}
$(E,h)$ is said to be Nakano-positive, if for any nonzero vector
$u=u^{i\alpha}\frac{\partial}{\partial z^i}\otimes e_\alpha,$
\begin{equation}
  \sum_{i,j,\alpha,\beta}R_{i\overline{j}\alpha\overline{\beta}}u^{i\alpha}\overline{u}^{j\beta}>0.\nonumber
\end{equation}
$(E,h)$ is said to be dual-Nakano-positive, if for any nonzero
vector $u=u^{i\alpha}\frac{\partial}{\partial z^i}\otimes e_\alpha,$
\begin{equation}
  \sum_{i,j,\alpha,\beta}R_{i\overline{j}\alpha\overline{\beta}}u^{i\beta}\overline{u}^{j\alpha}>0.\nonumber
\end{equation}

\edefinition

%

\bdefinition [cf. \cite{SS}] \label{k-posive} Let $M$ be a compact
complex manifold and $L\rightarrow M$ be a holomorphic line bundle
over $M$. \bd \item $L$ is called \emph{$k$-positive} ($0\leq k\leq
n-1$) if there exists a smooth Hermitian metric $h^L$ on $L$ such
that the curvature form $\sqrt{-1}\Theta(L, h^L)=-\sq\p\bp\log h^L$ is
semipositive everywhere and has at least $n-k$ positive eigenvalues
at every point of $M$.

\item $L$ is called \emph{$k$-ample} ($0\leq k\leq
n-1$), if $L$ is semi-ample and suppose that $L^m$ is globally
generated for some $m>0$, and the maximum dimension of the fiber of
the evaluation map $X\>\P\left(H^0(M,L^m)^*\right)$ is $\leq k$.

\ed

 \edefinition

\noindent Hence, the concepts of $0$-positivity, $0$-ampleness and
ampleness are the same. However, it is pointed out in
\cite[p.~127]{SS} that, $k$-ampleness is irrelevant to the metric
$k$-positivity when $k\geq 1$.

\subsection{Simple normal crossing divisors and Poincar\'{e} Type Metric}

On a compact K\"ahler manifold $X$, a divisor $D=\sum_{i=1}^s D_i$
 is  called a \emph{simple normal crossing divisor}  if every irreducible component  $D_i$ is smooth and all
intersections are transverse. The sheaf of germs of differential
$p$-forms on $X$ with at most logarithmic poles along $D$, denoted
{$\Om^p_X(\log D)$} ( introduced by Deligne in \cite{Del69}) is the
sheaf whose sections on an open subset $V$ of $X$ are \beq
\Gamma(V,\Om_X^p(\log D)):=\{\alpha\in \Gamma(V,\Om_X^p\ts\sO_X(D))\
\text{and }d\alpha\in \Gamma(V,\Om_X^{p+1}\ts\sO_X(D)) \}.\eeq

We will consider the complement $Y=X-D$ of a simple normal crossing divisor $D$ in a compact
K\"ahler manifold $X$.
It is well-known that we can choose a local coordinate chart
$(W;z_1,...,z_n)$ of $X$ such that the
    locus of $D$ is given by $z_1\cdots z_k=0$ and $Y\cap W=W_r^*=(\Delta_r^*)^k\times (\Delta_r)^{n-k}$
    where $\Delta_r$ (resp. $\Delta_r^*$) is the (resp. punctured) open disk of
    radius $r$ in the complex plane and $r\in (0,\frac{1}{2}]$. Instead of focusing on the
compact complex manifold $X$, we shall give a K\"{a}hler metric
$\omega_Y$ only on the open manifold $Y$, which enjoys some special
asymptotic behaviors along $D$.

\bdefinition \label{poincare}
 We say that the metric $\omega_Y$ on $Y$ is of Poincar\'{e} type along $D$, if for each local coordinate chart $(W;z_1,...,z_n)$
  along $D$ the restriction $\omega_Y|_{W_{\frac{1}{2}}^{*}}$ is equivalent to the usual Poincar\'{e} type metric $\omega_{P}$ defined by
\begin{equation}
  \omega_{P}=\sqrt{-1}\sum_{j=1}^k\frac{dz_j\wedge d\overline{z}_j}{|z_j|^2\cdot\log^2 |z_j|^2}+\sqrt{-1}\sum_{j=k+1}^{n}dz_j\wedge
  d\overline{z}_j.
\end{equation}
\edefinition

Two nonnegative functions or Hermitian metrics $f$ and $g$ defined
on $W_{\frac{1}{2}}^{*}$ are said to be equivalent along $D$ if for
any relatively compact subdomain $V$ of $W$, there is a positive
constant $C$ such that $(1/C)g\leq f\leq Cg$ on $V-D$. In this case
we shall use the notation $f\sim g$.

\vskip 0.4\baselineskip

 As a fundamental result along this line,  it is well-known
that, see \cite[Section~3]{Zucker22}, there always exists a K\"ahler
metric $\omega_Y$ on $Y=X-D$ which is of Poincar\'e type along $D$.
Furthermore, this metric  is complete and of  finite volume.
Moreover, it has bounded geometry which implies that its curvature tensor and covariant derivatives are
bounded. We will use these properties frequently in this paper. The
following model example is used in analyzing the integrability of
holomorphic sections with respect to the Poincar\'e type metrics.

\bexample\label{example} For any positive integer $n$, the integral
$$\int_0^{\frac{1}{2}} r^\alpha(\log r)^n dr$$ is finite if and
only if $\alpha>-1$.\eexample

\subsection{$L^2$-Estimates and $L^2$ Cohomology}
 We need the following $L^2$-estimates,
which will be used frequently in this paper. \blemma
[{\cite{AV01,Hormander09,Dem}}] \label{L2} Let $(M,\omega)$ be a
complete K\"{a}hler manifold. Let $(E,h^E)$ be an Hermitian vector
bundle over $M$. Assume that $A=[i\Theta(E,h^E),\Lambda_\omega]$ is
positive definite everywhere on $\Lambda^{p,q}T^{*}M\otimes E$,
$q\geq1$. Then for any form $g\in L^2(X,\Lambda^{p,q}T^{*}M\otimes
E)$ satisfying $\overline{\partial} g=0$ and $\int_M
(A^{-1}g,g)dV_\omega <+\infty,$
 there exists $f\in L^2(X,\Lambda^{p,q-1}T^{*}M\otimes E)$ such that $\overline{\partial}f=g$ and
 \begin{equation}
   \int_M|f|^2 dV_\omega\leq \int_M (A^{-1}g,g)dV_\omega.\nonumber
 \end{equation}
\elemma

\noindent Suppose that $\omega_Y$ is a smooth K\"ahler metric on $Y$
and $h^E_Y$ is a smooth Hermitian metric on $E|_Y$.  The sheaf
$\Omega_{(2)}^{p,q}(X,E,\omega_Y,{h}_Y^E)$ (for short
$\Omega_{(2)}^{p,q}(X,E)$) over $X$ is defined as follows. On any open
subset $U$ of $X$, the section space
$\Gamma(U,\Omega_{(2)}^{p,q}(X,E))$ of $\Omega_{(2)}^{p,q}(X,E)$
over $U$ consists of $E$-valued $(p,q)$-forms $u$ with measurable
coefficients such that the $L^2$ norms of both $u$ and
$\overline{\partial}u$ are integrable on any compact subset $V$ of
$U$. Here the integrability means that both $|u|_{\omega_Y\ts
h_Y^E}^2$ and $|\bp^{E} u|_{\omega_Y\ts h_Y^E}$ are integrable on
$V-D$.  It is well-known that the spaces of global sections of
$\Omega_{(2)}^{p,q}(X,E)$ with $\overline{\partial}$ operator form
an $L^2$ Dolbeault complex
\begin{equation}
  \Gamma(X,\Omega_{(2)}^{p,0}(X,E))\rightarrow \Gamma(X,\Omega_{(2)}^{p,1}(X,E))\rightarrow\cdot\cdot\cdot\rightarrow \Gamma(X,\Omega_{(2)}^{p,n}(X,E))\rightarrow0
\end{equation}
and the associated cohomology groups $H_{(2)}^{p,*}(Y,E)$ are called the $L^2$ Dolbeault cohomology groups with values in $E$.

Recall that a sheaf $\mathscr{S}$ over $X$ is called a \emph{fine sheaf} if for any finite open covering $\mathfrak{U}=\{U_j\}$, there is a family of homomorphisms $\{h_j\}$, $h_j:\mathscr{S}\to \mathscr{S}$, such that the support of $h_j$ satisfying that $\text{Supp}(h_j)\subset U_j$ and $\sum_j h_j=\text{identity}$ (cf. \cite[Definition 3.13]{Kodaira}). For any fine sheaf $\mathscr{S}$, one has $H^q(X, \mathscr{S})=0$ for $q\geq 1$ (cf. \cite[Theorem 3.9]{Kodaira}).

We have already known that if the
K\"{a}hler metric $\omega_Y$ on $Y$ is of Poincar\'{e} type along
$D$, then the K\"{a}hler metric $\omega_Y$ will be complete and with
finite volume (cf. \cite{Zucker22}). In this case if $u$ is an
$E$-valued $(p,q)$-form such that $u$ and $\overline{\partial}u$ are
$L^2$ local integrable on $U$ and if $f$ is a smooth function on
$X$, then it is obvious that both $fu$ and $\overline{\partial}(fu)$
will still be $L^2$ local integrable on $U$. Thus the sheaf $\Omega_{(2)}^{p,q}(X,E)$ admits a partition of unity and we conclude that
$\Omega_{(2)}^{p,q}(X,E)$ is a fine sheaf over $X$, so we have  $
H^{i}(X,\Omega_{(2)}^{p,q}(X,E))=0$, for $p,q \geq0$ and $i\geq1$.

\subsection{$\mb{R}$-divisors and $\mb{R}$-linear equivalence}
\noindent \\
For readers' convenience, we  explain the notions  in Theorem
\ref{main}. Let $X$ be a compact K\"ahler manifold.

(1).  $T$ is called an $\mb{R}$-divisor, if it is an element of
$\mathrm{Div}_{\mb{R}}(X):=\mathrm{Div}(X)\otimes_{\mb{Z}} \mb{R}$, where $\mathrm{Div}(X)$ is
the set of divisors in $X$. Two divisors $T_1,T_2$ in
$\mathrm{Div}_{\mb{R}}(X)$ are said to be linearly equivalent, denoted by
$T_1\sim_{\mb{R}} T_2$, if their difference $T_1-T_2$ can be written
as a finite sum of principal divisors with real coefficients, i.e.
 \begin{align}\label{d111}
   T_1-T_2=\sum_{i=1}^k r_i(f_i)
 \end{align}
where $r_i\in \mb{R}$ and $(f_i)$ is the principal divisor
associated to a meromorphic function $f_i$ (cf.
\cite[5.2.3]{Fujino}).

 (2). An $\mb{R}$-line bundle $L=\sum_i a_i L_i$ is a finite sum with
  some real numbers $a_1,\cdots, a_k$ and certain line bundles $L_1,\cdots,
L_k$.
Note that we also use $``\ts"$
for operations on line bundles. An $\R$-line bundle $L=\sum_i a_i
L_i$  is said to be $k$-positive if there exist smooth metrics
$h_1,\cdots, h_k$ on $L_1,\cdots, L_k$ such that the curvature of
the induced metric on $L$, which is explicitly given by
$$\sqrt{-1}\Theta(L, h)=\sum_{i=1}^k a_i \sqrt{-1}\Theta({L_i},
h_i)$$ is $k$-positive. It is easy to see that if there is another
expression $L=\sum_{i=1}^{\ell} b_i \tilde L_i$ for the $k$-positive
line bundle $L$, then by $\p\bp$-lemma on compact K\"ahler
manifolds, we can find smooth metrics on $\tilde L_i$ such that the
induced metric on $L$ is also $k$-positive. The definitions for
$\Q$-line bundles and $\Q$-divisors are similar.

\bremark As we shall see in the proofs of Theorem \ref{main} and its
applications, the Hermitian metrics  on $\R$-line bundles and
$\R$-divisors play the key role in the analytic approaches. \eremark

\vskip 2\baselineskip

\section{An $L^2$-type Dolbeault isomorphism}\label{haha}
\noindent In this section we will establish an $L^2$-type Dolbeault
isomorphism by using H\"{o}rmander's $L^2$-estimates. 

\btheorem \label{L2 iso11} Let $(X,\omega)$ be a compact \ka
manifold of dimension $n$ and $D$ be a simple normal crossing
divisor in $X$. Let $\omega_P$ be a smooth K\"{a}hler metric
 on $Y=X-D$ which is of Poincar\'{e} type along $D$. Then
there exists a smooth Hermitian metric $h_Y^L$ on $L|_Y$ such that
the sheaf $\Omega^{p}(\log D)\otimes\mathcal{O}(L)$ over $X$  enjoys
a fine resolution given by the $L^2$ Dolbeault complex
$(\Omega_{(2)}^{p,*}(X,L,\omega_P,{h}_Y^L),\overline{\partial})$,
that is, we have an exact sequence of sheaves over $X$
\begin{equation}\label{3.1}
  0\rightarrow\Omega^{p}(\log D)\otimes\mathcal{O}(L)\rightarrow\Omega_{(2)}^{p,*}(X,L,\omega_P,{h}_Y^L)
\end{equation}
such that $\Omega_{(2)}^{p,q}(X,L,\omega_P,{h}_Y^L)$ is a fine sheaf
for any $0\leq p,q\leq n$. In particular, \begin{equation}
  H^{q}(X,\Omega^{p}(\log D)\otimes\mathcal{O}(L))\cong H_{(2)}^{p,q}(Y,L,\omega_P,{h}_Y^L)\cong
  \mathbb{{H}}_{(2)}^{p,q}(Y,L,\omega_P,{h}_Y^L).\label{isoo}
\end{equation}
\etheorem

\bproof Let $h^L$ be an arbitrary smooth Hermitian metric on $L$
over $X$. Let $\sigma_i$ be the defining section of $D_i$. Fix
smooth Hermitian metrics $\|\bullet\|_{D_i}$ on $[D_i]$ such that
$\|\sigma_i\|_{D_i}<\frac{1}{2}$. For arbitrarily fixed constants
$\tau_i\in (0,1]$, we construct a smooth Hermitian metric
$h^{L}_{\alpha,\tau}:=h_Y^L$ on $L|_Y$ as
\begin{equation}\label{c1}
 h^{L}_{\alpha,\tau}=\prod_{i=1}^s\|\sigma_i\|_{D_i}^{2\tau_i}(\log^2(\|\sigma_i\|_{D_i}^{2}))^{\frac{\alpha}{2}}h^L.
\end{equation}
where $\alpha$ is a large positive constant (even integer) to be
determined later. It is well known that
$\Omega_{(2)}^{p,q}(X,L,\omega_P, h^L_Y)$ are fine sheaves over $X$
since $\omega_P$ on $Y=X-D$ is of Poincar\'{e} type along $D$,
so we only need to check the
exactness of (\ref{3.1}).

First let us consider the exactness of (\ref{3.1}) at $q=0$. Let
$(W;z_1,...,z_n)$ be a local coordinate chart of $X$ along $D$. Let
$e$ be a trivialization section of $L$ on $W$ such that
$\frac{1}{2}\leq|e(z)|_{h^L}\leq1$ over $W$. Denote
\begin{center}
  $\zeta_j=\frac{1}{z_j}dz_j$, for $1\leq j\leq t$; and $\zeta_j=dz_j$, for $t+1\leq j\leq n$.
\end{center}
Let $\sigma$ be a holomorphic section of $\Omega_{(2)}^{p,0}(X,L)$
on $W$. Then we can write
\begin{equation}
  \sigma(z)=\sum_{|I|=p}\sigma_I(z)\zeta_{i_1}\wedge\cdot\cdot\cdot\wedge\zeta_{i_p}\otimes e\nonumber
\end{equation}
where $I=(i_1,...,i_p)$ is a multi-index with
$i_1<\cdot\cdot\cdot<i_p$ and $\sigma_I(z)$ is a holomorphic
function on $W_{1/2}^{*}$. By definition, we see that $\sigma$ is
$L^2$ integrable on
$W_r^{*}\triangleq\Delta_{r}^{*t}\times\Delta_{r}^{n-t}\subset
W_{1/2}^{*}$ for any $0<r<1/2.$ Note that the Hermitian metric
$h^{L}_{\alpha,\tau}|_{W_{1/2}^{*}}$ is equivalent to the following
Hermitian metric
\begin{equation}
  h^{L}_{\alpha}=h^{L}_{\alpha}(W_{1/2}^{*})=\prod_{i=1}^t|z_i|^{2\tau_i}(\log^2|z_i|^{2})^{\frac{\alpha}{2}}h^L.
\end{equation}
If we denote $\{i_1,...,i_p\}\cap\{1,...,t\}=\{i_{p1},...,i_{pb}\}$,
then
\begin{equation}
  \|\sigma\|_{L^2(W_r^{*})}^2=\sum_{|I|=p}\int_{W_r^{*}} |e|_{h^L}^2
  \left(|\sigma_I(z)|^2 \prod_{\nu=1}^b\log^2|z_{i_{p\nu}}|^2
  \prod_{i=1}^t|z_i|^{2\tau_i}(\log^2|z_i|^{2})^{\frac{\alpha}{2}}
  \right)\omega_P^{n}.\label{key24}
\end{equation}
Suppose that the Laurent series representation of $\sigma_I(z)$ on
$W_{1/2}^{*}$ is given by
\begin{equation}
  \sigma_I(z)=\sum_{\beta=-\infty}^{\infty}\sigma_{I\beta}(z_{t+1},...,z_n)z_1^{\beta_1}\cdot\cdot\cdot z_t^{\beta_t}, \beta=(\beta_1,...,\beta_t)\nonumber
\end{equation}
where $\sigma_{I\beta}(z_{t+1},...,z_n)$ is a holomorphic function
on $\Delta_{1/2}^{n-t}$. Then by using polar coordinates and
Fubini's theorem (e.g. Example \ref{example}),  we see that $\sigma$
is $L^2$ integrable on $W_r^{*}$ if and only if  $\beta_j>-\tau_j$
along $D_j$. Since $\tau_j\in(0,1]$, we see $\beta_j\geq 0$ and
$\sigma_I(z)$ has removable singularity. Hence $\sigma$ and $\nabla
\sigma$ have only logarithmic pole, and $\sigma$ is a section of
$\Omega^{p}(\log D)\otimes\mathcal{O}(L)$ on $W$. Conversely, if we
choose $\sigma$ to be a holomorphic section of $\Omega^{p}(\log
D)\otimes\mathcal{O}(L)$ on $W$, it is easy to check by formula
(\ref{key24}) that $\sigma$ is $L^2$ integrable on $W_r^{*}$ for any
$0<r<\frac{1}{2}$. Therefore we have proved that (\ref{3.1}) is
exact at $\Omega_{(2)}^{p,0}(X,L)$ for any
$\alpha>0$.\\

Now we consider the exactness of (\ref{3.1}) at $q\geq1$. For any
fixed $r\in(0,1/2)$, we deform the K\"{a}hler metric $\omega_{P}$ to
be a new K\"{a}hler metric $\widetilde{\omega}_{P}$ on $W_r^{*}$,
given by
\begin{equation}
  \widetilde{\omega}_{P}=\widetilde{\omega}_{P}(W_r^{*})
  =\omega_{P}+\sqrt{-1}\sum_{i=1}^n\partial\overline{\partial}\psi_{i}
  =\sqrt{-1}\sum_{i=1}^n\widetilde{g}_{ii}dz_i\wedge d\overline{z}_i
\end{equation}
where $\psi_i(z)=\frac{1}{r^2-|z_i|^2}, z\in W_r^{*}$. Then it is
easy to check that $\widetilde{\omega}_{P}$ is a complete K\"{a}hler
metric on $W_r^{*}$. We define  a new Hermitian metric
$\widetilde{h}^L_{\alpha}$ for $L$ on $W_r^{*}$ as
\begin{equation}
 \widetilde{h}^L_{\alpha}=\widetilde{h}^L_{\alpha}(W_r^{*})
  =\prod_{i=1}^t|z_i|^{2\tau_i}(\log^2|z_i|^{2})^{\frac{\alpha}{2}}
  \prod_{i=1}^{n}\exp(-2\alpha|z_i|^2-\alpha\psi_{i})h^L.
\end{equation}

\blemma \label{1414} On $W_r^{*}$ the Chern curvature of
$\widetilde{h}^L_{\alpha}$ satisfies
\begin{equation} \label{2424}
  \sqrt{-1}\Theta(L, \widetilde{h}^L_{\alpha})\geq\alpha\widetilde{\omega}_{P}
\end{equation}
for some large $\alpha>0$. \elemma

\bproof  It is easy to show that \beq \sq\p\bp\log |z_i|^2=0
\qtq{and} -\sqrt{-1}
   \partial\overline{\partial}\log(\log (|z_j|^2))^2=\frac{2\sqrt{-1}dz_j\wedge
d\overline{z}_j}{|z_j|^2(\log(|z_j|^2))^2}.\label{computation}\eeq
The curvature of $(L,\widetilde{h}^L_{\alpha,\tau})$ is given by

\begin{equation}
\begin{split}
  \sqrt{-1}\Theta(L, \widetilde{h}^L_{\alpha})&=-\sqrt{-1}\sum_{i=1}^t\tau_i\partial\overline{\partial}\log|z_i|^2-\sqrt{-1}\frac{\alpha}{2}
   \sum_{i=1}^t\partial\overline{\partial}\log(\log^2 |z_i|^2)\\
&\hspace*{1.26cm}
+2\sqrt{-1}\alpha\sum_{i=1}^n\partial\overline{\partial}|z_i|^2+\sqrt{-1}\alpha\sum_{i=1}^n\partial\overline{\partial}
   \psi_i+\sqrt{-1}\Theta(h^L)\\
&\geq-\sqrt{-1}\frac{\alpha}{2}
   \sum_{i=1}^t\partial\overline{\partial}\log(\log^2|z_i|^2)+\sqrt{-1}\alpha\sum_{i=1}^n\partial\overline{\partial}|z_i|^2+\sqrt{-1}\alpha\sum_{i=1}^n\partial\overline{\partial}
   \psi_i\\
&\geq\alpha\widetilde{\omega}_{P},\nonumber
\end{split}\end{equation}
if we choose $\alpha$ large enough so that
$\sqrt{-1}\alpha\sum_{i=1}^n\partial
\overline{\partial}|z_i|^2+\sqrt{-1}\Theta(h^L)\geq0$ on $W_r^{*}$.
\eproof

\blemma\label{Na} On the chart $W_r^*$, the vector bundle  $
V:=\Om_Y^p\ts K_Y^{-1}\ts
  L|_Y$ with the induced metric $h^V$ by $\tilde \omega_P$ and $\tilde
  h_\alpha^L$ is Nakano positive when $\alpha$ is large enough.
  Moreover,  for any $u\in \Gamma(W^*_{\frac{r}{2}}, \Lambda^{n,q}T^*Y\ts
  V)$ we have
  \beq \langle\left[\sq\Theta(V, h^V),\Lambda_{\tilde\omega_P}\right]u,u\rangle\geq C|u|^2\label{ine}\eeq
  where $C$ is a positive constant independent of $u$.
\elemma

\bproof Note that the metric $\tilde \omega_P$ on the
holomorphic tangent bundle $TY$ is
  of the splitting form, i.e.
  \beq \tilde \omega_P= \sum_{i=1}^n\omega_i(z_i),\eeq
  and that the metric $\omega_i(z_i)$ depends only on the variable
$z_i$. Hence, by using curvature formula (\ref{cur}), in  local
computations, we can treat $(TY,\tilde \omega_P)$ as a direct sum of
line bundles $\ds_{i=1}^n(F_i,\omega_i)$. It is easy to check that the
curvature of $(F_i,\omega_i)$ \beq |\sq \p\bp\log \omega_i|\leq
C\tilde \omega_P\eeq for some positive constant $C$ independent of
$\alpha$. Hence, in local computations, the curvature of
$V=\Om_Y^p\ts K_Y^{-1}\ts
  L|_Y$ is the curvature of a direct sum of line bundles  $L|_Y\ts F^{-1}_{i_1}\ts\cdots\ts
  F^{-1}_{i_{n-p}}$. Therefore, by using the curvature estimate (\ref{2424}),  when $\alpha>(n-p+1)C$, the
  curvature of each summand   $L\ts F^{-1}_{i_1}\ts\cdots\ts
  F^{-1}_{i_{n-p}}$ is strictly positive. That means $V$ is Nakano
  positive. The inequality (\ref{ine}) follows from a
  straightforward
  calculation.
\eproof

\blemma The sequence of (\ref{3.1}) is exact at $q\geq1$.  That is,
on a small local chart
  $W^*_{\frac{r}{2}}=(\Delta_\frac{r}{2}^*)^k\times
(\Delta_\frac{r}{2})^{n-k}$, for any  $\bp$-closed  $L$-valued
$(p,q)$ form $\eta$ on $W_{\frac{r}{2}}^*$, if it is
$L^2$-integrable with respect to $(\omega_P, h^L_{\alpha})$, then
there exists an $L$-valued $(p,q-1)$ form $f$ on $W_{\frac{r}{2}}^*$
such that $f$ is $L^2$-integrable with respect to $(\omega_P,
h^L_{\alpha})$ and $\bp f= \eta$. \elemma

\bproof For simplicity, we write $W=W^*_{\frac{r}{2}}$. Suppose
$\eta\in\Gamma(W, \Lambda^{p,q}T^*Y\ts
  L)$ is $\bp_L$ closed and $L^2$-integrable with
  respect to $( \omega_P, h^L_{\alpha})$. Note that $V=\Om^p_Y\ts K^{-1}_Y\ts
  L|_Y$, we have
  \beq  \Gamma(W,\Lambda^{n,q}T^*Y\ts V)\cong \Gamma(W, \Lambda^{p,q}T^*Y\ts
  L|_Y).\eeq
   Since $\tilde h^L_{\alpha,\tau}\sim h^L_{\alpha,\tau}$, $\omega_P\sim \tilde \omega_P$ on $W$,   by Lemma \ref{Na} and Lemma \ref{L2},
   there exists $$f\in \Gamma(W,\Lambda^{n,q-1}T^*Y\ts V)\cong \Gamma(W, \Lambda^{p,q-1}T^*Y\ts
  L)$$ such that $\bp f=\eta$ on $W^*_{\frac{r}{2}}$, and $f$ is $L^2$- integrable with respect
  to $(\tilde \omega_P, \tilde h^L_{\alpha})$. By restricting to $W=W^*_{\frac{r}{2}}$, we have that $f$ is also $L^2$-integrable over $W$ with respect
  to $(\omega_P,  h^L_{\alpha})$.
\eproof

\noindent Given the exact sequence in (\ref{3.1}),  the isomorphisms
in (\ref{isoo}) are clear. The proof of Theorem \ref{L2 iso11} is
completed. \eproof

\bremark\label{re} \bd\item The isomorphism (\ref{isoo})  holds up
to equivalence of metrics. More precisely, if $\tilde \omega_P \sim
\omega_P$ and $\tilde h_Y^L\sim h_Y^L$, then
$$\mathbb{{H}}_{(2)}^{p,q}(Y,L,\omega_P,{h}_Y^L) \cong \mathbb{{H}}_{(2)}^{p,q}(Y,L,\tilde \omega_P,{\tilde h}_Y^L). $$

\item
 From the proof of Theorem \ref{L2 iso11}, it is easy to see that
the isomorphism in Theorem \ref{L2 iso11}  also works  for vector
bundles. \ed\eremark

\vskip 2\baselineskip

\section {Logarithmic  vanishing theorems}\label{V01}
\noindent In this section, we  prove  Theorem \ref{main} and several
applications described in the first section.

\begin{thm} [=Theorem \ref{main}] \label{a11}
 Let $X$ be a compact \ka manifold of dimension $n$ and $D=\sum_{i=1}^s D_i$ be a
simple normal crossing divisor in $X$.  Let $N$ be a line bundle and
$\t=\sum_{i=1}^s a_iD_i$ be an $\R$-divisor  with $a_i\in[0,1]$ such
that $N\otimes \mc{O}_X([\t])$ is a $k$-positive $\R$-line bundle.
Then for any nef line bundle $L$, we have
  \begin{align*}
    H^q(X,\Omega^p_X(\log D)\otimes L\otimes N)=0\quad \text{for any}\,\,\, p+q\geq n+k+1.
  \end{align*}
\end{thm}

\bproof  Let $\omega_0$ be a fixed K\"{a}hler metric on $X$. Let
$F=N\ts\sO_X([D])$. Since $F$ is a $k$-positive $\R$-line bundle,
there exist smooth metrics $h^N$ and $h^{[D_i]}$ on $F$ and $[D_i]$
respectively, such that the curvature form of the induced metric
$h^F$ on $F$
\begin{align}\label{b5}\sqrt{-1}\Theta(F,h^F)=\sq\Theta(N,h^N)+\sq\sum_{i=1}^s a_i \Theta([D_i], h^{[D_i]})\end{align}
is semipositive and has at least $n-k$ positive eigenvalues at each
point of $X$.

\vskip 0.3\baselineskip

 Let $\{\lambda^j_{\omega_0}(h^{F})\}_{j=1}^n$ be the eigenvalues
of $\sqrt{-1}\Theta(F, h^{F})$ with respect to $\omega_0$ such that
$\lambda^j_{\omega_0}(h^{F})\leq \lambda^{j+1}_{\omega_0}(h^{F})$
for all $j$.  Thus for any $j\geq k+1$ we have
 $$\lambda^j_{\omega_0}(h^{F})\geq \lambda^{k+1}_{\omega_0}(h^{F})\geq\min_{x\in X} \left(\lambda^{k+1}_{\omega_0}(h^{ F})(x)\right)=:c_0>0.$$
We set   $\delta=\frac{c_0}{32 n^2}$. Without loss of generality, we
 assume $\delta\in (0,1)$. Since $L$ is nef, there exists a smooth
metric $h^L_{\delta}$ on $L$ such that
\begin{align}\label{b3}\sqrt{-1}\Theta(L,h^L_{\delta})=-\sqrt{-1}\p\b{\p}\log h^L_{\delta}>-\delta\omega_0.\end{align}

\noindent Let $\sigma_i$ be the defining section of $D_i$. Fix
smooth metrics   $h_{D_i}:=\|\cdot\|_{D_i}^2$  on line bundles
$[D_i]$, such that $\|\sigma_i\|_{D_i}<\frac{1}{2}$. Write the curvature form of $[D_i]$ as $c_1(D_i) = \sqrt{-1}\Theta([D_i], h_{D_i})$.  We define
$h^{\t}:=\prod_{i=1}^s h_{D_i}^{a_i}$, then the curvature form of
$(\Delta, h^{\Delta})$ is
\begin{align}\label{4.3}
  -\sqrt{-1}\p\bp\log h^{\t}=-\sqrt{-1}\p\bp\log\prod_{i=1}^s h_{D_i}^{a_i}.
\end{align}

\noindent For simplicity, we set
  \beq\mathscr{F}:=L\ts N=L\otimes F\otimes \mc{O}_X(-[\t]).\eeq The induced metric on $\mathscr{F}$ is defined by
  \begin{align*}
    h^{\mathscr F}_{\alpha,\epsilon,\tau}=h^L_{\delta}\cdot h^F\cdot (h^{\t})^{-1}\cdot\prod_{i=1}^s \|\sigma\|^{2\tau_i}_{D_i}\left(\log^2 (\epsilon\|\sigma_i\|^2_{D_i})\right)^{\frac{\alpha}{2}}.
  \end{align*}
  Here the constant $\alpha>0$ is chosen to be large enough and the constants $\tau_i,\epsilon \in (0,1]$ are to be determined
  later. Note that the smooth metric $h^F\cdot (h^{\t})^{-1}$ on $N=F\otimes
  \mc{O}_X(-[\t])$ is the same as $h^N$ up to a globally defined
  function over $X$.
  A straightforward computation shows that
  \begin{align}\label{b4}
  \begin{split}
    &\sqrt{-1}\Theta\left(\mathscr{F},h^{\mathscr{F}}_{\alpha,\epsilon,\tau}\right)\\
    &=\sqrt{-1}\Theta(F,h^F)+\sqrt{-1}\Theta(L,h^L_{\delta})+\sum_{i=1}^{s}(\tau_i-a_i)c_1(D_i)\\
    &\quad +\sum_{i=1}^s \frac{\alpha c_1(D_i)}{\log (\epsilon\|\sigma_i\|^2_{D_i})}+\sqrt{-1}\sum_{i=1}^s\frac{\alpha \p\log\|\sigma_i\|^2_{D_i}\wedge \b{\p}\log\|\sigma_i\|^2_{D_i}}{(\log (\epsilon\|\sigma_i\|^2_{D_i}))^2}.
    \end{split}
  \end{align}

\noindent  Since $a_i\in [0,1]$, for a fixed large $\alpha$,  we can
choose $\tau_1, \cdots, \tau_{s}\in (0,1]$ and $\epsilon$ such that
$\tau_i-a_i$, $\epsilon$
  are small enough and
  \begin{align}\label{b2}
  -\frac{\delta}{2}\omega_0\leq \sqrt{-1}\sum_{i=1}^s(\tau_i-a_i)c_1(D_i)\leq
  \frac{\delta}{2}\omega_0, \ \ \  -\frac{\delta}{2}\omega_0\leq\sum_{i=1}^s \frac{\alpha c_1(D_i)}{\log (\epsilon\|\sigma_i\|^2_{D_i})}\leq \frac{\delta}{2}\omega_0.
  \end{align}
Note that the constants $\tau_i$ and $\epsilon$ are thus fixed, and
the choice of $\epsilon$ depends on $\alpha$. \noindent We set \beq
\omega_Y=\sqrt{-1}\Theta\left(\mathscr{F},h^{\mathscr{F}}_{\alpha,\epsilon,\tau}\right)+2(4n+1)\delta\omega_0.\eeq
It is easy to check that $\omega_Y$ is a Poincar\'{e} type K\"ahler
metric on $Y$. \noindent By (\ref{b5}), (\ref{b3}), (\ref{b4}) and
(\ref{b2}), one has on $Y$
  \begin{align}\label{b6}
    \sqrt{-1}\Theta\left(\mathscr{F},h^{\mathscr{F}}_{\alpha,\epsilon,\tau}\right)\geq \sqrt{-1}\Theta(F,h^{F})-2\delta \omega_0.
  \end{align}
  Since $\sqrt{-1}\Theta(F,h^{F})$ is a semipositive (1,1) form, we see that on $Y$
  \begin{align}\label{b1}
    \omega_Y=\sqrt{-1}\Theta\left(\mathscr{F},h^{\mathscr{F}}_{\alpha,\epsilon,\tau}\right)+2(4n+1)\delta\omega_0\geq 8n\delta\omega_0.
  \end{align}
This implies that
  \begin{align*}
  \begin{split}
  \sqrt{-1}\Theta\left(\mathscr{F},h^{\mathscr{F}}_{\alpha,\epsilon,\tau}\right)=\omega_Y-2(4n+1)\delta\omega_0
  \geq-\frac{1}{4n}\omega_Y.
  \end{split}
  \end{align*}
By exactly the same argument as in the proof of Theorem \ref{L2
iso11} (see also Remark \ref{re}), when $\alpha$ is large enough, we
obtain
  \beq  H^q\left(X,\Omega^p_X(\log D)\otimes \mathscr{F}\right)\cong H^{p,q}_{(2)}\left(Y, \mathscr{F}, \omega_Y,
  h^{\mathscr{F}}_{\alpha,\epsilon,\tau}\right). \label{isooo} \eeq

\vskip0.3\baselineskip Next, we prove the vanishing of  the $L^2$
cohomology groups by using Lemma \ref{L2}.
 On a local chart of $Y$, we may assume that
$\omega_0=\sqrt{-1}\sum_{i=1}^n \eta_i\wedge \b{\eta}_i$ and
  $$\sqrt{-1}\Theta\left(\mathscr{F},h^{\mathscr{F}}_{\alpha,\epsilon,\tau}\right)=\sqrt{-1}\sum_{i=1}^n \lambda^i_{\omega_0}\left(h^{\mathscr{F}}_{\alpha,\epsilon,\tau}\right)\eta_i\wedge \b{\eta}_i.$$
  Then
  \begin{align*}
    \sqrt{-1}\Theta\left(\mathscr{F},h^{\mathscr{F}}_{\alpha,\epsilon,\tau}\right)
    &=\sqrt{-1}\sum_{i=1}^n \lambda^i_{\omega_0}\left(h^{\mathscr{F}}_{\alpha,\epsilon,\tau}\right)\eta_i\wedge \b{\eta}_i\\
    &=\sqrt{-1}\sum_{i=1}^n\frac{\lambda^i_{\omega_0}\left(h^{\mathscr{F}}_{\alpha,\epsilon,\tau}\right)}{\lambda^i_{\omega_0}\left(h^{\mathscr{F}}_{\alpha,\epsilon,\tau}\right)+2(4n+1)\delta}\eta'_i\wedge \b{\eta}'_i\\
    &=\sqrt{-1}\sum_{i=1}^n\frac{16n^2\lambda^i_{\omega_0}\left(h^{\mathscr{F}}_{\alpha,\epsilon,\tau}\right)}{16n^2\lambda^i_{\omega_0}\left(h^{\mathscr{F}}_{\alpha,\epsilon,\tau}\right)+(4n+1)c_0}\eta'_i\wedge \b{\eta}'_i
  \end{align*}
where $$\eta'_i=\eta_i\cdot
\sqrt{\lambda^i_{\omega_0}(h^{\mathscr{F}}_{\alpha,\epsilon,\tau})+2(4n+1)\delta}.$$
Note
 that $\omega_Y=\sqrt{-1}\sum_{i=1}^n \eta'_i\wedge
\b{\eta}'_i, $ and so the eigenvalues of
$\sqrt{-1}\Theta\left(\mathscr{F},h^{\mathscr{F}}_{\alpha,\epsilon,\tau}\right)$
with respect to $\omega_Y$ are
  \begin{align*}
    \gamma_i:=\frac{16n^2\lambda^i_{\omega_0}\left(h^{\mathscr{F}}_{\alpha,\epsilon,\tau}\right)}{16n^2\lambda^i_{\omega_0}\left(h^{\mathscr{F}}_{\alpha,\epsilon,\tau}\right)+(4n+1)c_0}<1.
  \end{align*}
Thus $\gamma_j\in [-\frac{1}{4n},1)$. On the other hand, by
(\ref{b6}) one has
  \begin{align*}
    \lambda_{\omega_0}^{j}\left(h^{\mathscr{F}}_{\alpha,\epsilon,\tau}\right)\geq \lambda_{\omega_0}^{j}(h^{ F})-2\delta.
  \end{align*}
Hence for any $j\geq k+1$, we have
  \begin{align*}
    \lambda_{\omega_0}^{j}\left(h^{\mathscr{F}}_{\alpha,\epsilon,\tau}\right)\geq
    \min_{x\in X}  \left(\lambda_{\omega_0}^{k+1}(h^{F})(x)\right)-2\delta
    =c_0-2\delta=\left(1-\frac{1}{16n^2}\right)c_0>0.
  \end{align*}
  It also implies that for $j\geq k+1$,
  \begin{align*}
    \gamma_j&=\frac{16n^2\lambda^i_{\omega_0}\left(h^{\mathscr{F}}_{\alpha,\epsilon,\tau}\right)}{16n^2\lambda^i_{\omega_0}\left(h^{\mathscr{F}}_{\alpha,\epsilon,\tau}\right)+(4n+1)c_0}\\
    &\geq \frac{16n^2(1-\frac{1}{16 n^2})c_0}{16n^2(1-\frac{1}{16 n^2})c_0+(4n+1)c_0}=1-\frac{1}{4n}.
  \end{align*}
For any section $u\in \Gamma(Y,\Lambda^{p,q}TY\ts \mathscr{F})$, we
obtain
\begin{align*}
  \left\langle\left[\sqrt{-1}\Theta\left(\mathscr{F}, h^{\mathscr{F}}_{\alpha,\epsilon,\tau}\right),\Lambda_{\omega_Y}\right]u,u\right\rangle
  &\geq \left(\sum_{i=1}^q \gamma_i
  -\sum_{j=p+1}^n \gamma_j\right)|u|^2\\
  &\geq \left((q-k)\left(1-\frac{1}{4n}\right)
  -\frac{k}{4n}-(n-p)\right)|u|^2\\
  &=\left((q+p-n-k)-\frac{q-k}{4n}-\frac{k}{4n}\right)|u|^2\\
  &\geq \frac{1}{2}|u|^2.
\end{align*}
Thus, Theorem \ref{a11} follows from (\ref{isooo}) and Lemma
\ref{L2}.\eproof

\noindent As applications of Theorem \ref{a11}, we obtain

\bcorollary  \label{c11}Let $X$ be a compact \ka manifold of
dimension $n$ and $D$ be a simple normal crossing divisor.  Suppose
that $N$ is a $k$-positive line bundle and $L$ is a nef line bundle,
 then
\begin{equation}
H^{q}(X, \Om_X^p(\log D)\ts N\ts L)=0 \qtq{for any $p+q\geq
n+k+1$.}\nonumber
\end{equation}
\ecorollary

\noindent In particular, one can deduce  the following well-known
result.

\bcorollary\label{c22} Let $X$ be a compact \ka manifold of
dimension $n$ and $D$ be a simple normal crossing divisor. Suppose
that $L\rightarrow X$ is an ample line bundle, then
\begin{equation}
H^{q}(X, \Om_X^p(\log D)\ts L)=0\qtq{for any $p+q\geq n+1$.}\nonumber
\end{equation}
\ecorollary

\noindent As an analogue to Corollary \ref{c22}, we  obtain  the
following log  type Le Potier vanishing theorem for ample vector
bundles.

\bcorollary \label{c44} Let $X$ be a compact \ka manifold of
dimension $n$ and $D$ be a simple normal crossing divisor. Suppose
that $E\rightarrow X$ is an ample vector  bundle of rank $r$. Then
\begin{equation}
H^{q}(X, \Om_X^p(\log D)\ts E)=0\qtq{for any $p+q\geq n+r$.}\nonumber
\end{equation}
 \ecorollary

 \bproof Let $\pi:\P(E^*)\>X$ be the projective bundle of $E$ and $\sO_E(1)$ be the tautological line bundle. By using the Le Potier
isomorphism (e.g. \cite[Theorem~5.16]{SS}), we have\beq
H^{q}(X,\Om_X^{p}(\log D)\ts E)\cong
H^q(\P(E^*),\Om^p_{\P(E^*)}(\log\pi^{*}D)\ts \sO_E(1)).\eeq On the
other hand, it is easy to see that $\pi^{-1}D$ is also a simple
normal crossing divisor. Hence, Corollary \ref{c44} follows from
Corollary \ref{c11}. \eproof

\noindent By using the same strategy as in the proof of Theorem
\ref{a11}, we also obtain several log type Nakano
 vanishing theorems for vector bundles on $X$. For instance,

\bproposition\label{Nakano} Let $E$ be a vector bundle of rank $r$
and $L$ be a line bundle on $X$.

\bd\item  If $E$ is Nakano positive $($resp. Nakano semi-positive$)$
and $L$ is nef $($resp. ample$)$, then for any $q\geq 1$
\begin{equation}
H^q(X,\Om_X^n(\log D)\otimes E\ts L)=0.\nonumber
\end{equation}

\item If $E$ is dual-Nakano positive $($resp. dual-Nakano semi-positive$)$ and  $L$ is nef $($resp. ample$)$, then for any $p\geq 1$
\begin{equation}
H^n(X,\Om^p_X(\log D)\otimes E\ts L)=0.\nonumber
\end{equation}

\item If $E$ is globally generated and  $L$ is ample, then for any $p\geq 1$
\begin{equation}
H^n(X,\Om^p_X(\log D)\otimes E\ts L)=0.\nonumber
\end{equation}
\ed \eproposition

\noindent Indeed, the vector bundle $E\ts L$ in Proposition
\ref{Nakano} is either Nakano positive or dual Nakano positive (e.g.
\cite{LY15}). Hence, the proof is very similar to $($but simpler
than$)$ that in Theorem \ref{a11}.

\section{Applications }

 In this section, we  present several straightforward  applications
 of Theorem \ref{main} over compact K\"ahler manifolds,
 which are also closely related to a number of classical vanishing theorems in algebraic
 geometry.

\begin{theorem}\label{cor4.7} Let $X$ be a compact \ka manifold of
dimension $n$ and $D=\sum_{i=1}^s D_i$ be a simple normal crossing
divisor. Suppose $F$ is a line bundle over $X$ and $m$ is a positive
real number such that $mF=L+D'$, where $D'=\sum_{i=1}^s \nu_i D_i$
is an effective normal crossing $\R$-divisor and  $L$ is a
$k$-positive $\R$-line bundle. Then
  \begin{align}
    H^q\left(X,\Omega^p(\log D)\otimes F\otimes \mc{O}_X\left(-\sum_{i=1}^s\left(1+\left[\frac{\nu_i}{m}\right]\right)D_i\right)\right)=0
  \end{align}
  for $p+q\geq n+k+1$.
\end{theorem}
\begin{proof}  Let  $$N=F\otimes
\mc{O}_X\left(-\sum_{i=1}^s\left(1+\left[\frac{\nu_i}{m}\right]\right)D_i\right)$$
and
$$\Delta=\sum_i\left(1+\left[\frac{\nu_i}{m}\right]-\frac{\nu_i}{m}\right)D_i.$$
We have that \beq N\ts \sO_X([\Delta])=\frac{1}{m} L,\eeq which is a
$k$-positive $\R$-line bundle. Hence we can apply Theorem \ref{a11}
to complete the proof.
\end{proof}

\begin{corollary} \label{cKV}
Let $X$ be a compact K\"{a}hler manifold $D=\sum_{j=1}^sD_j$ be a
simple normal crossing divisor of $X$. Let $[D']$ be a $k$-positive
$\mb{R}$-line bundle over $X$, where $D'=\sum_{i=1}^s c_i D_i$ with
$c_i>0$ and $c_i\in\mb{R}$. Then
$$H^q(X,\Omega^p(\log D)\otimes \mc{O}_X(-\lceil D'\rceil))=0 ~~\text{for any}~~ p+q<n-k.$$In particular, when $[D']$ is  ample, \beq H^q(X,\Omega^p(\log
D)\otimes \mc{O}_X(-\lceil D'\rceil))=0,\ ~~\text{for}~~ p+q<n. \eeq

\end{corollary}
\begin{proof} Let $$N=
\mc{O}_X(-D)\otimes\lceil D'\rceil, \qtq{and}
\Delta=\sum_i\left(1+c_i-\lceil c_i\rceil\right)D_i.$$ It is easy to
see that \beq N\ts\sO_X([\Delta])=[D']\eeq which is a $k$-positive
$\R$-line bundle. By using Theorem \ref{a11}, one has
  \begin{equation}\label{TT}
    H^q(X,\Omega^p(\log D)\otimes \mc{O}_X(-D)\otimes\lceil D'\rceil)=0
  \end{equation}
for any $p+q\geq n+k+1$. By Serre duality and the isomorphism \beq
(\Omega^p_X(\log D))^*\cong\Omega^{n-p}_X(\log D)\otimes
\mc{O}_X(-K_X-D), \label{sd}\eeq we see that (\ref{TT}) is equivalent to
$$H^q(X,\Omega^p(\log D)\otimes \mc{O}_X(-\lceil D'\rceil))=0$$
for any $p+q<n-k$. The proof is complete.
\end{proof}

\bcorollary Let $X$ be a compact K\"{a}hler manifold
and $D=\sum_{j=1}^sD_j$ be a simple normal crossing divisor of $X$. Let
$[D']$ be a $k$-positive $\mb{R}$-line bundle over $X$, where
$D'=\sum_{i=1}^s a_i D_i$ with $a_i>0$ and $a_i\in\mb{R}$.
 If
there exists a line bundle $L$ over $X$ and a real number $b$ with
$0< a_j<b$ for all $j$,
 and $bL=[D']$ as $\R$-line bundles, then
  $$H^q(X,\Omega^p(\log D)\otimes L^{-1})=0$$
   for $p+q>n+k$ and $p+q<n-k$.

\ecorollary \bproof Let $b'$ be a real number such that $\max_j
a_j<b'<b$ and set
  $$N=L^{-1},\quad
  \Delta=\frac{D'}{b'}=\sum_{j=1}^s\frac{a_j}{b'}D_j.$$ Let
$$F=L^{-1}\ts \sO_X([D])=L^{-1}+\frac{D'}{b'}=\frac{b-b'}{bb'}D'.$$
  It is easy to see that $F$ is a $k$-positive $\R$-line bundle and the coefficients of $\Delta$ are in $[0,1]$.
 By Theorem \ref{a11}, we obtain $$H^q(X,\Omega^p(\log D)\otimes L^{-1})=0
  \qtq{
   for $p+q>n+k.$}$$

  On the other hand,  we can set
  $$N=L\ts \sO_X(-D),\ \ \  \Delta=\sum_{j=1}^s\left(1-\frac{a_j}{2b}\right)D_j, \qtq{and} F=N\ts \sO_X([D])=\frac{D'}{2b}.$$
  It is easy to see that $F$ is a $k$-positive $\R$-line bundle and the coefficients of $\Delta$ are in
  $[0,1]$. By Theorem \ref{a11} again, we get
  $$H^{q}(X,\Omega^p(\log D)\otimes L\otimes \mc{O}_X(-D))=0, \qtq{for
  $p+q>n+k.$}$$
By Serre duality and the isomorphism (\ref{sd}), we have
$$H^q(X,\Omega^p(\log D)\otimes L^{-1})=0$$ for any $p+q<n-k$.
\eproof

\bcorollary Let $X$ be a compact \ka manifold of dimension $n$ and
$D=\sum_{i=1}^s D_i$ be a simple normal crossing divisor in $X$.
Suppose there exist some real constants $a_i\geq 0$ such that
$\sum_{i=1}^s a_iD_i$ is a $k$-positive $\R$-divisor, then for any
nef line bundle $L$, we have
  \begin{align*}
    H^q(X,\Omega^p_X(\log D)\otimes L)=0\quad \text{for any}\,\,\, p+q\geq n+k+1.
  \end{align*}
\ecorollary

\bproof We can set $N=\sO_X$ and $\Delta=\frac{1}{1+\sum_{i=1}^sa_i}\sum_{i=1}^s a_iD_i$. Then
$N\ts\sO([\Delta])=\sO([\Delta])$ is a $k$-positive $\R$-line
bundle. \eproof

\end{document}